\documentclass[12pt]{amsart}
\usepackage[utf8]{inputenc}
\usepackage[cal=boondoxo]{mathalfa}
\usepackage{amssymb}
\usepackage{comment}
\usepackage{amsmath}
\usepackage{amsthm}
\usepackage{graphicx}
\usepackage{mathtools}
\usepackage{wrapfig}
\usepackage{xcolor}
\usepackage{verbatim}
\usepackage{indentfirst}
\usepackage{enumerate}
\usepackage{a4}
\usepackage{graphics}
\usepackage{epsfig}
\usepackage[colorlinks=true,linkcolor=black,citecolor=black,urlcolor=black]{hyperref}
\newtheorem{thm}{Theorem}[section]
\newtheorem{lm}[thm]{Lemma}
\newtheorem{ex}[thm]{Example}

\newtheorem{pr}[thm]{Proposition}

\numberwithin{equation}{section}

\newcommand{\R}{\mathbb{R}}
\newcommand{\D}{\mathbb{D}}
\newcommand{\C}{\mathbb{C}}

\newcommand{\norm}[1]{\Vert #1\Vert}

\newcommand{\Dn}[1]{\mathbb{D}^{#1}}
\newcommand{\Cn}[1]{\mathbb{C}^{#1}}
\newcommand{\Bn}[1]{\mathbb{B}^{#1}}

\newcommand{\br}[1]{\{#1\}}
\newcommand{\Fa}{F_{\alpha}}
\newcommand{\Fto}[2]{F_{#1,#2}}
\newcommand{\aut}[1]{\text{Aut}(#1)}

\begin{document}
\thispagestyle{plain}
\begin{center}
    \Large
    \textbf{The von Neumann inequality for 3$\times$3 matrices in the unit Euclidean ball}

    \vspace{0.4cm}
    \text{Dariusz Piekarz}\footnote{Partially supported by the Partially supported by the National Center of Science, Poland, Preludium Bis 2 grant no. 2020/39/O/ST1/00866}
\end{center}       
 \vspace{0.9cm}
\textbf{ABSTRACT}\\
It is shown that the constant $c_{d,3}$ in von Neumann's inequality for $d$-tuples of commutative and row contractive $3\times3$ matrices, as proved by Hartz, Richter, and Shalit in \cite{1}, is independent of the size of the $d$-tuple. A numerical estimation of the constant is provided.

\vspace{0.4cm}
\section{Introduction}
The von Neumann inequality is a classical result in operator theory which states that, for any bounded and contractive operator $T\in B(H)$ acting on a Hilbert space $H$ and any polynomial $p\in\C[z]$, the following inequality holds: \begin{align}\label{cvn:1}\norm{p(T)}\leq\sup_{z\in\D}|p(z)|.\end{align}
It is a natural question to ask whether this inequality can be extended to a $d$-tuple of operators. For $d=2,$ And$\hat{\text{o}}$ answered to this question positively proving that
if $T=(T_{1},T_{2})$ is a pair of commuting, contractive, bounded operators on a Hilbert space $H$ and $p\in\C[z,w]$ is a polynomial, then  $$\norm{p(T)}\leq\sup_{(z,w)\in\Dn{2}}|p(z,w)|.$$

However, a full generalization of this result is not possible. Davie and Crabb proved that in a 8-dimensional Hilbert space there exist a triple $T$ of pairwise commuting, contractive operators, and a homogeneous polynomial $p$ of degree 3 such that  $$\norm{p(T)}>1.$$ Further details on the proof of And$\hat{\text{o}}$'s inequality and the Davie-Crabb example can be found in Chapter 1 of \cite{6}.
\par In \cite{4}, Knese used And$\hat{\text{o}}$'s inequality and Kosiński's result on a solution of the 3-point Pick interpolation problem in the polydisc $\Dn{d}$ to show that the von Neumann inequality holds for $d$-tuples of $3\times 3$ commuting and contractive matrices in the polydisc $\Dn{d}$.
\par Recall that a Pick interpolation problem, for a domain $D\subset\Cn{n}$, can be formulated as follows: given distinct points $z_1,\ldots,z_N$ in $D$ and points $\zeta_1,\ldots,\zeta_N$ in the unit disc $\D$, decide if there exists a holomorphic function $F:D\rightarrow\D$ such that $F(z_j)=\zeta_j,\ j=1,\ldots,N$.
\par
Additionally, the Pick interpolation is extremal if it is solvable, and there is no holomorphic function $G:D\rightarrow\D$ such that $G(z_j)=\zeta_j,\ j=1,\ldots,N$ and $G(D)$ is a relatively compact subset of $\D.$
\par  In the work of Hartz, Richter and Shalit \cite{1} the contractivity property was replaced by the row contractivity condition, which means that a $d$-tuple of operators $T$ defined on a Hilbert space satisfies the condition
\begin{align}\label{rc:1}
\sum_{j=1}^{d}T_{j}T^{*}_{j}\leq I,
\end{align}where $I$ represents the identity operator, and $T^{*}$ denotes the adjoint operator of $T$. They proved: 
 \begin{thm}\label{vNb:1}There exists the smallest constant $c_{d,n}\geq0,$ such that
     for any $d-$tuple of commuting, row contractive, $n\times n$ matrices and for any polynomial $p\in\C[z_{1},...,z_{d}],$ one has 
     \begin{align}\label{vNb:2}
         \norm{p(T)}\leq c_{d,n}\sup_{z\in\Bn{d}}|p(z)|,
     \end{align}
     where $d,n\geq2.$ Additionally, $c_{d,n}>1$ if $n\geq3,\ d\geq2$ and $c_{d,2}=1,$ if $d\geq2.$
  \end{thm}

Following Knese's idea we use the solution to the Pick interpolation problem in the unit Euclidean ball, developed by Kosiński and Zwonek in \cite{3}, to prove Theorem \ref{vNb:1} in the particular case:

\begin{thm}\label{main_thm:1}
The constant $c_{d,3}$, described in Theorem \ref{vNb:1}, is independent of $d$, i.e., $$c_{d,3}=c_{2,3},\ d\geq2.$$
Moreover, one has the following numerical estimation: $1.11767\leq c_{d,3}\leq 3.14626.$
\end{thm}

\section{Case of a $d$-tuple of $2\times 2$ commuting, row contractive matrices}
To present our proof idea for Theorem \ref{main_thm:1}, we show a special case of Theorem \ref{vNb:1}, which was proved by Hartz, Richter, and Shalit in \cite{1}:

\begin{pr}\label{vN2:1}
The von Neumann inequality (\ref{vNb:2}) holds for any $d$-tuple of commuting, row contractive $2\times2$ matrices, and for any polynomial $p\in\C[z_{1},\dots,z_{d}],$ where $d\geq2$ with a constant $c_{d,2}=1$.
\end{pr}

The following lemma is crucial in the proof of Proposition \ref{vN2:1}:

\begin{lm}\label{NP2:1}
Let
\begin{align}\label{np2:2}
\Bn{d}\ni z\mapsto\zeta\in\D,\\
\Bn{d}\ni w\mapsto\eta\in\D\notag,
\end{align}
be a 2-point extremal Pick interpolation problem in the unit Euclidean ball. Then, up to composition with automorphisms of the unit Euclidean ball $\Bn{d}$ and the unit disk $\D$, it can be solved by the projection onto the first coordinate, $\pi(z_{1},\dots,z_{d})=z_{1}$.
\end{lm}
\begin{proof}[Proof of Lemma~\ref{NP2:1}]
Let $F$ be a solution of (\ref{np2:2}). It is known that automorphisms of the unit ball $\Bn{d}$ are finite compositions of transformations of the form $$A_{t}(z_{1},...,z_{d})=\left( \sqrt{1-|t|^{2}}\frac{z_{1}}{1+\overline{t}z_{d}},....,\sqrt{1-|t|^{2}}\frac{z_{d-1}}{1+\overline{t}z_{d-1}},\frac{z_{d}+t}{1+\overline{t}z_{d}}  \right),\ t\in\D$$
as well as unitary mappings. Using automorphisms of the unit ball $\Bn{d}$ and the unit disk $\D$, we can assume that $w=(0,\dots,0)$, $\eta=0$, and $\zeta>0$. Using additionally a unitary mapping, we can assume that $z=(x,0,\dots,0)$, where $x>0$. Consider a function $G:\D\rightarrow\D,$ given by $G(\mu):=F(\mu,0,\dots,0)$. It follows from the Schwarz lemma that $|G(\mu)|\leq|\mu|.$ Therefore, one has $\zeta\leq x$. Extremality means in fact that $\zeta=x$. The Schwarz lemma yields that the projection $\pi:\Bn{d}\rightarrow\D$ given by  $\pi(z_{1},\dots,z_{d})=z_{1}$ is the desired solution.
\end{proof}

 \begin{proof}[Proof of Proposition~\ref{vN2:1}]
The set of diagonalizable $d$-tuples of $2\times2$ matrices is dense in the set of all $d$-tuples of square $2\times 2$ matrices \cite[Paragraph 4, p.21]{7}, so we can assume that the $d$-tuple $T$ of pairwise commuting, row contractions consists of diagonalizable matrices.  In \cite[Theorem 2.21.]{5}, it was shown that any family of diagonalizable and commuting matrices is simultaneously diagonalizable. Hence, we can express $T$ in the form $$T=P\text{diag}(\alpha,\beta)P^{-1},$$ where $\text{diag}(\alpha,\beta)$ is defined as the $d$-tuple of diagonal matrices
\begin{align*}
\text{diag}(\alpha,\beta)&:= \left( \begin{pmatrix}
\alpha_{1} &0\\
0 & \beta_{1}
\end{pmatrix},...,\begin{pmatrix}
\alpha_{d} & 0\\
0 & \beta_{d}
\end{pmatrix} \right),\end{align*} where $\alpha = (\alpha_{1},...,\alpha_{d})\in\Bn{d},\ \beta = (\beta_{1},...,\beta_{d})\in\Bn{d}$
 and $P$ is a transformation matrix.
Consequently, for any holomorphic function $f:\Bn{d}\rightarrow\C$, we obtain
$$ f(T) = P \text{diag}(f(\alpha),f(\beta))P^{-1}.$$
In particular, this holds for polynomials.

Consider a polynomial $p$ and a $d$-tuple $T$ of commutative, row-contractions with a joint spectrum $\br{\alpha,\beta}$ in $\Bn{d}$. We can assume that $p(\alpha)\neq p(\beta).$ Otherwise $T=p(\alpha)I,$ where $I$ is the identity matrix. One has in this case $\norm{p(T)}=|p(\alpha)|\leq\sup_{\lambda\in\Bn{d}}|p(\lambda)|$ and the assertion follows. Next, let the interpolation problem
\begin{align}\label{pr:1}
\alpha\mapsto p(\alpha),\notag\\
\beta\mapsto p(\beta),
\end{align}
have an analytic solution $f$. Consider the modified problem,
\begin{align}\label{pr:2}
\alpha\mapsto tp(\alpha),\notag\\
\beta\mapsto tp(\beta),
\end{align}
where $t\geq1$. We are interested in finding the greatest $t$ such that there exists a solution of (\ref{pr:2}). Note that the set $\tau$ of all such $t$ is bounded from above as the image of any solution of (\ref{pr:2}) is a subset of the unit disc $\D$.
Let $(t_{i})_{i}$ be a sequence converging to the supremum  $t_{0}$ of $\tau.$ For any $i$, there exists a holomorphic function $f_{i}:\Bn{d}\rightarrow\D$ such that
\begin{align}\label{pr:3}
f_{i}(\alpha)= t_{i}p(\alpha),\ f_{i}(\beta)=t_{i}p(\beta).
\end{align}
Using Montel's theorem, we can pass to a convergent subsequence. Therefore, there exists a holomorphic function $F:\Bn{d}\rightarrow\D$ such that $$F(\alpha)= t_{0}p(\alpha),\  F(\beta)=t_{0}p(\beta).$$ Since $t_{0}$ is the greatest possible, the problem is extremal. Due to Lemma \ref{NP2:1}, the function $F$ can be expressed in the form $F=m\circ \pi\circ A,$ where $m\in\aut{\D},\ A\in\aut{\Bn{d}}$ and $\pi$ is the natural projection onto the first coordinate. Put $S=A(T)$. Since a composition of the unit ball automorphism with a $d$-tuple of commutative, row contractions is a commutative, row contraction, we infer that $S$ is a $d$-tuple of commutative, row contractions. It follows that $$F(T)=m(\pi(A(T)))=m(S_{1}).$$ We have\begin{align}\label{fin:1}\norm{p(T)}\leq\norm{F(T)}=\norm{m(S_{1})}\leq\sup_{z\in\D}|m(z)|=1.\end{align} In the last inequality we applied inequality (\ref{cvn:1}) which holds for all holomorphic functions in the neighbourhood of the unit disc $\D$ due to uniform approximation by complex polynomials on compact subsets of the unit disc $\D$. Thus, we have $c_{d,2}=c_{1,2}$.
\end{proof}
 
\section{Proof of Theorem \ref{main_thm:1}: Constant $c_{d,3}=c_{2,3}$}
\par  In \cite{3} Kosiński and Zwonek found a solution to the 3-point Pick interpolation problem in the unit Euclidean ball.
\begin{thm}[Kosiński, Zwonek]\label{kzw:1}
If the $3-$point Pick interpolation problem
$$\Bn{d}\rightarrow\D,\ w_{j}\mapsto\lambda_{j},\ j = 1, 2, 3,$$
is extremal, then, up to a composition with automorphisms of $\Bn{d}$ and $\D$, it is
interpolated by a function which belongs to one of the classes
$$\mathcal{F}_{D}=\br{(z_{1},...,z_{n})\mapsto\frac{2z_{1}(1-\tau z_{1})-\overline{\tau}\omega^{2}z_{2}^{2}}{2(1-\tau z_{1})-\omega^{2}z_{2}^{2}}:\ |\tau|=1,|\omega|\leq1},$$ 
$$\mathcal{F}_{ND}=\br{(z_{1},...,z_{n})\mapsto\frac{z_{1}^{2}+2\sqrt{1-a^{2}}z_{2}}{2-a^{2}}:\ a\in[0,1)}.$$
\end{thm}
We begin with showing that for every integer $d\geq2$, it is true that $c_{d,3}=c_{2,3}$. To prove it, we can adopt the same reasoning employed in the $2\times 2$ matrix case, but instead consider $d$-tuples of diagonalizable $3\times3$ matrices (the fact that the set of $d-$tuples of $3\times3,$ diagonalizable and commuting matrices is dense in set of all $d-$tuples of commuting square matrices can be found in \cite[Lemma 10]{8}). Assuming that $T$ consists of a $d$-tuple of pairwise commuting, diagonalizable, row contractive matrices, we can represent $T$ as $$T=P\text{diag}(\alpha,\beta,\gamma)P^{-1},$$ where $P$ is a transformation matrix and $\alpha,\beta,\gamma\in\Bn{d}.$ As before, let $p$ be a polynomial. First, consider the case where $p(\alpha),p(\beta),p(\gamma)$ are distinct. The interpolation problem
\begin{align*}
\alpha\mapsto p(\alpha),\notag\\
\beta\mapsto p(\beta),\\
\gamma\mapsto p(\gamma),\notag
\end{align*}
possesses an analytic solution $f$. Next, we consider the modified problem,
\begin{align}\label{pr:4}
\alpha\mapsto tp(\alpha),\notag\\
\beta\mapsto tp(\beta),\\
\gamma\mapsto tp(\gamma).\notag
\end{align} Using the same Montel-type argument as in the proof of Proposition \ref{vN2:1}, there exists a holomorphic function $F:\Bn{d}\rightarrow\D$ that is an extremal solution of the interpolation problem (\ref{pr:4}).
Theorem \ref{kzw:1}  allows us to represent  $F$ in the form  $F=m\circ h\circ A,$ where $m\in\aut{\D},\ A\in\aut{\Bn{d}}$ and $h$ is a function belonging either to $\mathcal{F}_{D},$ or $ \mathcal{F}_{ND}.$ Put $S=A(T).$ Then $$F(T)=m(h(A(T)))=m(h(S_{1},S_{2})),$$ due to the special forms of functions from classes $\mathcal{F}_{D},\mathcal{F}_{ND}.$ 
\par As in the case of $2\times 2$ matrices, if $p(\alpha)=p(\beta)=p(\gamma),$ then we can proceed as in the proof of Proposition \ref{vN2:1}.\par Next, consider the case where only 2 of 3 points $p(\alpha), p(\beta), p(\gamma)$ are equal, say $p(\alpha)=p(\beta).$ Then, we can modify the polynomial $p$ replacing $p(\beta)$ by $p(\beta)-\varepsilon\eta,$  where $\varepsilon>0$ is arbitrarily small number and $\eta\in\partial\D$ are such that $p(\beta)-\varepsilon\eta\in\D,$ for all such $\varepsilon>0$. Denote the modified polynomial by $p_{\varepsilon}.$ Points of spectrum for $p_{\varepsilon}(T)$  are distinct, hence we can argue as above constructing a holomorphic function $F_{\varepsilon}$ such that $F_{\varepsilon}(T)=m(h_{\varepsilon}(A(T)))=m(h_{\varepsilon}(S_{1},S_{2})).$ Since this is true for all such $\varepsilon>0$ suitably small, by the continuity of the norm and the continuity of holomorphic functions we can pass with $\varepsilon$ to $0$ preserving required properties. To get back to the von Neumann's inequality we argue similarly as in the end of the proof of the Proposition \ref{vN2:1}. Therefore, we infer that $c_{d,3}=c_{2,3}.$

\section{Proof of Theorem \ref{main_thm:1}: Estimation of the $c_{2,3}$ constant}

According to the previous section, we can assume that $d=2$. Furthermore, we have a pair $T=(T_{1},T_{2})$ consisting of diagonalizable $3\times 3$ commuting and row contractive matrices. We can assume that the joint spectrum of $T$ is a set $\br{(a_{1},a_{2}),(0,d_{2}),(0,0)}\subset\Bn{2}$ with $d_{2}>0$.
Indeed,  the composition of an automorphism of the unit ball $\Bn{2}$ with a commutative row contraction is itself a commutative row contraction. Therefore, as shown in the proof of Lemma \ref{NP2:1}, we can compose $T$ with appropriate automorphisms of the unit ball $\Bn{2}.$

Using \cite[Theorem 2.23.]{5}, we can assume that $T_{i}$'s are upper triangular matrices, that is, $$T_{1}=\begin{pmatrix}
a_{1} & b_{1} & c_{1}\\
0 & 0 & e_{1}\\
0 & 0 & 0\
\end{pmatrix},\ T_{2}=\begin{pmatrix}
a_{2} & b_{2} & c_{2}\\
0 & d_{2} & e_{2}\\
0 & 0 & 0\
\end{pmatrix}.$$ 
We can represent the matrices $T_i$ using a block matrix notation \begin{align}\label{block:1}
    T_{i}=\begin{pmatrix}
A_{i} & \beta_{i} \\
0 & 0\ 
\end{pmatrix},
\end{align} where $$A_{1}=\begin{pmatrix}
a_{1} & b_{1} \\
0 & 0\
\end{pmatrix},\ A_{2}=\begin{pmatrix}
a_{2} & b_{2} \\
0 & d_{2}\
\end{pmatrix},\ \beta_{i}=\begin{pmatrix}
c_{i} \\
e_{i} \
\end{pmatrix}\ i=1,2.$$ 
As in the preceding sections, in order to estimate $c_{2,3}$, it is sufficient to appraise the norm of $(m\circ h\circ A)(T)$, where $m\in\aut{\D}$, $A\in\aut{\Bn{2}}$, and $h\in\mathcal{F}_{ND}$ or $h\in\mathcal{F}_{D}$.\par We split the estimation with respect to the function type.
\begin{lm}\label{fnd_simp:1}
If $h\in\mathcal{F}_{ND},$ then $\norm{h(T)}\leq2/\sqrt{3}.$
\end{lm}
\begin{proof}
Consider a function $\Fa$ from the class $\mathcal{F}_{ND}.$ It is worth noting that in this case, the inequality given by $(\ref{rc:1})$ can be restated in an equivalent form as $\norm{T_{1}z}^{2}+\norm{T_{2}z}^{2}\leq1,$ where $z\in\partial\Bn{3}.$  This can also be expressed using the block form as follows: 
\begin{align}\label{block_rc:1}
\norm{A_{1}x+\beta_{1}y}^{2}+\norm{A_{2}x+\beta_{2}y}^{2}\leq1,\ 
z=\begin{pmatrix}
x\\
y\ 
\end{pmatrix}\in\partial\Bn{3},\ x\in\Bn{2},\ y\in\D.
\end{align}
We get $$\Fa(T_{1},T_{2})=\begin{pmatrix}
\Fa(A_{1},A_{2}) & \frac{1}{2-\alpha^2}(A_{1}\beta_{1}+2\sqrt{1-\alpha^2}\beta_{2})\\
0 & 0\ 
\end{pmatrix}.$$
Using the triangle inequality and boundedness of $A_{1}$ we get
\begin{align}\label{fa_estim:1}
\norm{\Fa(T_{1},T_{2})\begin{pmatrix}
x\\
y\ 
\end{pmatrix}}^{2}&\leq\\ &\frac{1}{(2-\alpha^{2})^{2}}\left[\norm{A_{1}x+\beta_{1}y}+2\sqrt{1-\alpha^2}\norm{A_{2}x+\beta_{2}y}\right]^{2}\notag.
\end{align}
We can simplify the estimation of (\ref{fa_estim:1}) by making the substitution $u=\norm{A_{1}x+\beta_{1}y}$ and $v=\norm{A_{1}x+\beta_{1}y}$. Then, the estimation reduces to finding the supremum
$\sup\br{\frac{1}{2-\alpha^{2}}(u+2\sqrt{1-\alpha^2}v):\ 0\geq u,v,\ u^{2}+v^{2}\leq1,\ 0\leq\alpha<1}.$
This supremum equals $2/\sqrt{3}$.
    
\end{proof}
\par Next, we prove
\begin{lm}\label{fnd_aut:1}
If $h\in\mathcal{F}_{ND},$ and $m\in\aut{\D},$ then $\norm{(m\circ h)(T)}\leq3.14626.$
\end{lm}
\begin{proof}

Let $m_{t}\in\aut{\D},\ F_{\alpha}\in\mathcal{F}_{ND},$ where $t\in\D,\ \alpha\in[0,1).$
We compute
$$(m_{t}\circ\Fa)(T)=\begin{pmatrix}
m_{t}(\Fa(A)) & \frac{1}{2-\alpha^2}(\overline{t}m_{t}(\Fa(A))+I)(A_{1}\beta_{1}+2\sqrt{1-\alpha^2}\beta_{2})\\
0 & -t\ 
\end{pmatrix}.$$
Express $z\in\partial\Bn{3}$ in the block form $z=\begin{pmatrix}
x\\
y\ 
\end{pmatrix}, $ where $ x\in\Bn{2},\ y\in\D.$ Then, by the boundedness property of linear operators
\begin{align}\label{main:1}
    \norm{m_{t}(\Fa(T_{1},T_{2}))\begin{pmatrix}
x\\
y\ 
\end{pmatrix}}^{2}&\leq(\norm{m_{t}(\Fa(A_{1},A_{2}))}\norm{x}\\ &+\frac{|y|}{2-\alpha^{2}}\norm{\overline{t}m_{t}(\Fa(A_{1},A_{2}))+I}\norm{A_{1}\beta_{1}\notag\\ &+2\sqrt{1-\alpha^2}\beta_{2}})^{2}+|ty|^{2}. \notag
\end{align}

According to Proposition \ref{vN2:1}, we have that 
 $$\norm{m_{t}(\Fa(A_{1},A_{2}))}\leq1.$$ By applying the triangle inequality, we can estimate the right-hand side of inequality (\ref{main:1}) by
 \begin{align}\label{main:2}
     (\norm{x}+\frac{2|y|}{2-\alpha^{2}}\norm{A_{1}\beta_{1}+2\sqrt{1-\alpha^2}\beta_{2}})^{2}+|y|^{2}.
 \end{align}

We may utilize numerical computations to obtain an upper bound for the expression (\ref{main:2}). Because of the maximum modulus principle, we may also assume that $\norm{x}^{2}+|y|^{2}=1$. Hence, the right-hand side of inequality (\ref{main:1}) may be estimated from above by a function that is dependent on the coefficients of matrices $T_1$ and $T_2$, as well as the parameters $y$ and $\alpha$. All variables and parameters satisfy the constraints given by the equations  $|a_{1}|^{2}+|b_{1}|^{2}+|c_{1}|^{2}+|c_{2}|^{2}\leq1,\ |e_{1}|^{2}+|e_{2}|^{2}\leq1,\ |y|\leq1,\ 0\leq\alpha<1.$
By imposing these constraints on our function, we obtain $$\norm{m_{t}(\Fa(T_{1},T_{2}))}\leq3.14626.$$
    
\end{proof}

Finally, we consider the composition of a unit disc automorphism with a function from the class $\mathcal{F}_{D}$. We prove
\begin{lm}\label{fd_aut:1}
If $h\in\mathcal{F}_{D}$ and $m\in\aut{\D},$ then $\norm{(m\circ h)(T)}\leq3.14626.$
\end{lm}
\begin{proof}

\par Consider a function $\Fto{\tau}{\omega}\in\mathcal{F}_{D}$. Note that it can be expressed as $$\Fto{\tau}{\omega}(z,w)=\overline{\tau}F_{1,1}(\tau z,\omega w),$$ where $|\tau|\leq1,\ |\omega|=1.$ 
The function $F_{1,1}$ can be obtained as a limit of functions from the class $\mathcal{F}_{ND}$, composed with appropriate automorphisms of the unit disc $\D$ and the unit ball $\Bn{2}$. In particular, we have $$F_{1,1}(z,w)=\lim_{\R\cap\D\ni t\rightarrow 1^{-}}(m_{t}\circ F_{0}\circ A_{t}\circ u)(z,w),$$ where
\begin{itemize}
    \item $m_{t}(z)=-\frac{z-t}{1-tz}\in \aut{\D},\ t\in(0,1),$ 
    \item $A_{t}(z,w)=\large(\sqrt{1-t^{2}}\frac{z}{1+tw},\frac{w+t}{1+tw}\large)\in \aut{\Bn{2}},\ t\in(0,1),$
    \item $u(z,w)=(z,-w)$ is an unitary map,
    \item $F_{0}\in\mathcal{F}_{ND}.$
\end{itemize}
\par Consider a function $\Fto{\tau}{\omega}\in\mathcal{F}_{D}$, where $\tau,\omega\in\partial\D$. Let  $m_{s}(z)=\frac{z-s}{1-\overline{s}z}\in\aut{\D},$ for some $s\in\D,$ and $A\in\aut{\Bn{2}}.$ Then, the mapping $U_{\tau,\omega}:\Bn{2}\ni (z,w)\mapsto(\tau z, \omega w)\in\Bn{2}$ is unitary. 
\par Recall that the composition of a commutative, row contraction with a unitary map or a unit Euclidean ball automorphism is itself a commutative, row contraction. Therefore, if $T$ is a commutative, row contraction, then so is $U_{\tau,\omega}(T)$. In estimating $\norm{(m_{s}\circ\Fto{\tau}{\omega}\circ A)(T)}$ we are considering all pairs $T$ of commuting, row contractions. Hence, we can replace any composition of a unitary map or a unit ball automorphism with $T$ by $T$ itself. Thus, we have $$\norm{(m_{s}\circ\Fto{\tau}{\omega}\circ A)(T)}=\norm{(m_{s}\circ\overline{\tau}\Fto{1}{1}\circ U_{\tau,\omega}\circ A)(T)}= \norm{(m_{s}\circ\overline{\tau}\Fto{1}{1})(T)},$$ where $s\in\D,$ and $\omega,\tau\in\partial\D.$ Hence, using the continuity of a norm, we can estimate the norm of $(m_{s}\circ\overline{\tau}\Fto{1}{1})(T)$ as follows.
\begin{align}\label{fin:1}
    \norm{(m_{s}\circ\overline{\tau}\Fto{1}{1})(T)}&=\lim_{\R\cap\D\ni t\rightarrow 1^{-}}\norm{(m_{s}\circ\overline{\tau} m_{t}\circ F_{0}\circ A_{t}\circ u)(T)}\\ &= \lim_{\R\cap\D\ni t\rightarrow 1^{-}}\norm{(m_{\Tilde{s}}(F_{0}(T)))}\leq3.14626,\notag
\end{align} where $\Tilde{s}\in\D$ is such that $m_{\Tilde{s}} = m_{s}\circ\overline{\tau} m_{t}.$
Thus we have the upper bound on the constant $c_{2,3}$ for $\tau,\omega\in\partial\D$.
\par Consider the mapping $\D\ni\omega\mapsto (m_{s}\circ\Fto{\tau}{\omega})(T)$. Using the estimation (\ref{fin:1}) and the maximum modulus principle, we obtain
$$\norm{(m_{s}\circ\overline{\tau}\Fto{1}{1})(T)}\leq3.14626,$$ for $\omega\in\overline{\D}$ and $\tau\in\partial\D.$ The proof of the Lemma \ref{fd_aut:1} is completed.
\end{proof}

 \par Moreover, by computational simulations, one can obtain a lower bound estimate. The sample considered for estimation of the lower bound had a size $\sim 5\cdot10^{6}$ per each of classes $\mathcal{F}_{ND}$ and $\mathcal{F}_{D}.$
 \begin{ex}\label{ex:1} Let $\alpha=0.707107$ and 
 $$D_{1}=\textup{diag}(0.289 + 0.31i,0.02 - 0.08i,-0.22 - 0.12i),$$
 $$\ D_{2}=\textup{diag}(0.008 + 0.18i,-0.03 - 0.08i,0.08 + 0.13i),$$
 be diagonal matrices and let
$$P=\begin{pmatrix}
 0.1 - 0.15i& 0.15 + 0.58i &0.48 + 0.4i \\
-0.01 - 0.67i & 0.11 + 0.53i &0.17 + 0.64i \\
0.18 + 0.69i & -0.07 - 0.57i &-0.26 - 0.3i \
\end{pmatrix}.
 $$
 be a transition matrix.
 Consider the tuple $T=(PD_{1}P^{-1},PD_{2}P^{-1}).$ Then, $T$ consists of commutative, row contractions and
 $\norm{F_{\alpha}(T)}=1.11767,$ where $ F_{\alpha}\in\mathcal{F}_{ND}.$
 \end{ex}
 It is worth to notice that the best numerical hits were found for functions from the class $\mathcal{F}_{ND}.$ For the other class the highest norm we found was $\sim1.06.$ This let us believe that the optimal estimate of $c_{2,3}$ is the estimate we found in the Lemma \ref{fnd_simp:1} i.e. $2/\sqrt{3}\approx 1.1547.$
 It would be of interest to compare the estimations presented above with those obtained by Hartz, Richter, and Shalit in their work \cite{1}.  

\section*{Acknowledgments}
I would like to express my sincere gratitude to the anonymous reviewer for their valuable comments and suggestions, which significantly contributed to improving the quality of this article.

\address{DOCTORAL SCHOOL OF EXACT AND NATURAL SCIENCE, JAGIELLONIAN UNIVERSITY, LOJASIEWICZA 11, 30-348, KRAKOW, POLAND}
\address{INSTITUTE OF MATHEMATICS, FACULTY OF MATHEMATICS AND COMPUTER SCIENCE, JAGIELLONIAN UNIVERSITY, LOJASIEWICZA 6, 30-348, KRAKOW, POLAND}
\email{dariusz.piekarz@doctoral.uj.edu.pl}

\end{document}